\documentclass[11pt, a4paper]{amsart}
\usepackage{amsfonts,amsmath, amsthm, amssymb, amscd}
\input{amssym.def}

\usepackage[dvips]{graphicx}
\usepackage{epsfig}
\usepackage{psfrag}

\newcommand{\tri}{\;\;\makebox[0pt]{$|$}\makebox[0pt]{$\cap$}\;\;}

\theoremstyle{plain}
\newtheorem{theorem}{Theorem}[section]
\newtheorem{lemma}[theorem]{Lemma}
\newtheorem{cor}[theorem]{Corollary}

\newtheorem{dfn}[theorem]{{Definition}}
\theoremstyle{definition}
\newtheorem*{remo}{Remark}

\numberwithin{equation}{section}

\newcommand {\Bn}{\mathbb{N}}
\newcommand {\Bz}{\mathbb{Z}}
\newcommand {\Br}{\mathbb{R}}
\newcommand {\Bq}{\mathbb{Q}}

\begin{document}

\title[Characterization of geodesic flows on $T^2$]
{Characterization of geodesic flows on $T^2$ with and without positive topological entropy}

\author{Eva Glasmachers and Gerhard Knieper }
\date{\today}
\address{Faculty of Mathematics,
Ruhr University Bochum, 44780 Bochum, Germany}
\email{eva.glasmachers@rub.de, gerhard.knieper@rub.de}
\subjclass[2000]{Primary 37C40, Secondary 53C22, 37C10}
%53C44
\keywords{topological entropy, geodesic flows on tori}
%%%%%%%%%%
% final version, 14.6.2010
%%%%%%%%%%%%%%%%%%%%%%%%%%%%%%%%%%%%%%%%%%%%%%%%%%%%%%%%%%%%%%%%%%%%%%

\begin{abstract}
In the present work we consider the behavior of the geodesic flow on the unit
tangent bundle of the 2-torus $T^2$ for an arbitrary Riemannian metric.
A natural non-negative quantity which measures the complexity of the geodesic
flow is the topological entropy. In particular, positive topological entropy
implies chaotic behavior on an invariant set in the phase space of positive
Hausdorff-dimension (horseshoe). We show that in the case of zero topological
entropy the flow has properties similar to integrable systems. In particular
there exists a non-trivial continuous constant of motion which measures the
direction of geodesics lifted onto the universal covering $\Br^2$. Furthermore,
those geodesics travel in strips bounded by Euclidean lines. Moreover we derive
necessary and sufficient conditions for vanishing topological entropy involving
intersection properties of single geodesics on $T^2$.
\end{abstract}

%%%%%%%%%%%%%%%%%%%%%%%%%%%%%%%%%%%%%%%%%%%%%%%%%%%%%%%%%%%%%%%%%%%%%%

\maketitle

%%%%%%%%%%%%%%%%%%%%%%%%%%%%%%%%%%%%%%%%%%%%%%%%%%%%%%%%%%%%%%%%%%%%%%

\section{Introduction}

Let $(T^2,g)$ be a two-dimensional Riemannian torus. By $c_v$ we denote the
unique geodesic $c_v:\Br \to T^2$ with the initial condition $\dot c_v(0)=v \in ST^2$.
The geodesic flow on the unit tangent bundle $ST^2$ is given by $\phi^t(v)=\dot c_v(t)$.
We also consider geodesics on $\Br^2$, where $\Br^2$ is equipped with the lifted metric.
The topological entropy  of a continuous dynamical system represents the exponential
growth rate of orbits segments distinguishable with arbitrarily fine but finite precision.
It therefore describes the total exponential orbit complexity by a single number.
Note that due to a theorem of A. Katok \cite{K80} for $C^{1+ \alpha}$-flows, $\alpha>0$,
on 3-dimensional spaces positive topological entropy and the existence of a horse-shoe
are equivalent. We consider in this paper the questions which consequences on the behavior
of geodesics on $T^2$ we can expect under the assumption of zero topological entropy and which
geometrical restrictions on their behavior forbid high complexity of the geodesic flow.
More precisely, we formulate necessary and sufficient conditions for zero topological
entropy.
In order to state the main theorems we have to define the asymptotic direction of
geodesic rays, i.e.\ geodesics  $c:[0, \infty) \to \Br^2.$ Note that unlike in most
text books in our context a geodesic ray does not have to be minimal.
Given a geodesic $c:\Br \to \Br^2$ then we associate to $c$ two geodesic rays given
by $c^+:=c|_{[0, \infty)}: [0, \infty) \to \Br^2$ and
$c^-:[0, \infty) \to \Br^2$ with $c^-(t):=c(-t)$.\\
\\
In the following it will be useful to distinguish the following types of geodesic rays
(see also \cite{W}).\\

\begin{dfn} [types of geodesic rays] ~\\
Given a complete Riemannian metric on $\Br^2$. A geodesic ray $c: [0, \infty) \to \Br^2$
is called
\begin{itemize}
\item
{\it bounded}, if the set $c([0, \infty))$ is bounded.
\item
{\it escaping}, if $\lim\limits_{t \to \infty} \|c(t)\|= \infty$, where $\| \;\cdot \;  \|$
denotes the Euclidean norm  on $\Br^2$.
\item
{\it oscillating}, if $c$ is neither bounded nor escaping.
\end{itemize}
A geodesic $c: \Br \to \Br^2$ is bounded (escaping or oscillating) if both of its geodesic
rays $c^+$ and $c^-$ are bounded (escaping or oscillating).
\end{dfn}

\begin{remo}
It will turn out that mixed cases of geodesics will not be relevant.
Furthermore, their Liouville measure is zero, as proved by M. Wojtkowski in \cite{W}.\\
\end{remo}

\begin{dfn}[asymptotic direction and rotation number] ~\\
Let $c : [0, \infty) \rightarrow \Br^2$ be an escaping geodesic ray on $\Br^2$.
Then, if the limit exists,
\begin{align*}
    \delta(c):= \lim_{t \rightarrow \infty} \frac{c(t)}{\|c(t)\|} \in S^1 \enspace \quad
\end{align*}
is called the {\it asymptotic direction} of $c$.
 Let  $$\pi: S^1= \{(x,y) \mid x^2 + y^2 =1 \} \to \mathbb{P}_1(\Br) \cong \Br \cup \{\infty\},$$
 be the canonical projection onto $\mathbb{P}_1(\Br)$
 defined by
 $$
 \pi(s) = \left\{ \begin{array}{cc}
 \frac{y}{x}, & \text{if} \; x \not=0\\
 \infty, & \text{otherwise} \enspace .
 \end{array} \right.
 $$
 Then we call
\begin{align*}
    \rho(c): = \pi \circ \delta(c) \in  \mathbb{P}_1(\Br) \cong \Br \cup \{\infty\} \enspace \quad
\end{align*}
the {\it rotation number} of $c$. \\

Let $c: [0, \infty) \to T^2$ be a geodesic ray such that its lift $\tilde c: [0, \infty) \to \Br^2$
provides an escaping geodesic ray for which $\delta(\tilde c)$ exists.
 Then we define $\delta(c):=\delta(\tilde c)$ and $\rho(c):=\rho(\tilde c)$.
\\
A rotation number $\rho(c)$ is called {\it rational} if
$ \rho(c)  \in \Bq \cup \{\infty\}$. A direction $\delta(c) \in S^1$ is called {\it rational}
if $\pi(\delta(c)) \in \Bq \cup \{\infty\}$.\\
\end{dfn}

\begin{remo}~\\
\begin{enumerate}
\item[(a)]
The definition
and the existence of the asymptotic direction are independent of the chosen lift.
\item[(b)]
A first definition of a rotation number for minimal geodesics induced by the slope of the
accompanying Euclidean lines goes back to G.~A.~Hedlund (see \cite{Hedlund}) and H.~M.~Morse
(see \cite{Morse}). In~\cite{B1988} V.~Bangert presents a definition of the rotation
number for minimal geodesics, which coincides with our definition restricted to minimal geodesics.
\end{enumerate}
\end{remo}

Now we are able to state our first theorem.\\

\noindent{\bf Theorem~I.}
{\it Let $g$ be a Riemannian metric on $T^2$ with vanishing topological
entropy. Then, on the universal covering $\Br^2$, every geodesic $c$ is escaping without
self-intersections and the directions $\delta(c^+)$ and $\delta(c^-)$ exist. Furthermore,
$\delta(c^+)=-\delta(c^-)$ and $c$ lies in a bounded Euclidean strip $S(c)$, where $S(c)$
is a strip in $\Br^2$ bounded by two Euclidean lines. }\\
\\
Consider the asymptotic direction and the rotation number for a geodesic $c_v: \Br \to T^2$
as functions on the unit tangent bundle defined as $\delta(v):=\delta(c_v^+)$ and
$\rho(v):= \rho(c_v^+)$, if they exist. Obviously $\delta(c_v^-)=\delta(-v)$ and $\rho(c_v^-)=\rho(-v)$.\\

\noindent{\bf Theorem~II.}
{ \it Let $g$ be a Riemannian metric on $T^2$ with vanishing topological entropy.
Then the asymptotic direction $\delta: ST^2 \to S^1$ and
the rotation number $\rho : ST^2 \to \Br \cup \{\infty\}$
are surjective continuous functions invariant under the geodesic flow.
}\\

\begin{remo}
If the constant of motion $\rho$ would be differentiable with non-zero differential
almost everywhere, the system would be integrable in the sense of Liouville-Arnold.\\
\end{remo}

>From now on we represent $T^2$ as $\mathbb{R}^2/\mathbb{Z}^2$, where $\mathbb{Z}^2$ acts
on $\mathbb{R}^2$ via
$$
\tau_{(m,n)}(x,y) =(x+m,y+n) \enspace .
$$
The non-trivial covering transformations $\tau_{(m,n)}$ will be called translation elements.
Sometimes we will denote $\tau_{(m,n)}$ only by $\tau$.  A translation element $\tau$
is called primitive if it is not a nontrivial power of another translation element.
Two translation elements $\tau$ and $\eta$ are called equivalent if there exist
$k, \ell \in \mathbb{Z}\setminus\{0\}$ such that
$$
\tau^k=\eta^{\ell} \enspace.
$$
By $[\tau]$ we denote the corresponding equivalence classes.
\begin{dfn}
 A geodesic
$c: \Br \to \Br^2$ is called an axis if there exists a nontrivial translation element $\tau$ such
that $\tau c(t)=c(t+l)$ for some $l \in \Br$ and all $t \in \Br$. By $\delta(\tau)$ we
will denote the asymptotic direction of the axes of $\tau$. An axis $\alpha$ of a translation
element $\tau^k$ with $k \geq 2$ is called a non-primitive axis if $\alpha(\Br) \not= \tau \alpha(\Br)$.
A geodesic $c: \Br \to T^2$ is called prime-periodic if its lift is an axis of a primitive translation
element $\tau$. 
\end{dfn}

\begin{dfn}
For a geodesic ray $c: [0, \infty) \to \Br^2$ we define
$$
I(c) := \{ [\tau] \mid \# \{c([0, \infty)) \tri \eta c([0, \infty)) \} = \infty
\text{ for some $\eta \in [\tau]$}\}.
$$
\end{dfn}
We will now present a further characterization of the behavior of geodesics on the universal
covering in the case of vanishing topological entropy. By intersections we always mean
transversal intersections.\\

\noindent{\bf Theorem~III.}
{\it Let $(T^2,g)$ be a Riemannian torus with zero topological entropy.
Then $\# I(c) \leq 1$, for each geodesic ray  $c: [0, \infty) \to \Br^2$.
Furthermore, geodesics $c$ with irrational rotation number intersect their translates
$\tau c$ only a finite number of times. Geodesics $c$ with rational rotation number
intersect their translates $\tau c$ an infinite number of times at most for $\tau$ with
$\delta(\tau)=\delta(c^+)$.
}\\
\\
In a previous version of this paper we have had a weaker formulation of the following theorem.
We originally showed applying hyperbolic dynamics and the Curve Shortening flow that the
assumption of the theorem implies zero topological entropy. V.~Bangert showed to us how to
use variational arguments to even conclude flatness. We also note that we obtained flatness
under the stronger assumption that no geodesic on the universal covering intersects its
translate transversally.\\

\noindent{\bf Theorem~IV.}
{\it Let $g$ be a Riemannian metric on $T^2$. Then flatness of the metric $g$ is equivalent
to the condition that no axis $c: \Br \to \Br^2$ on the universal covering intersects
any of its translates.}\\

\begin{remo}~\\
\begin{enumerate}
\item[1)]
The assumption of Theorem~IV is equivalent to the fact that all axes $c: \Br \to \Br^2$ are axes of
primitive elements.
\item[2)]
As the theory of minimal geodesics and the Curve Shortening flow extend to symmetric Finsler metrics,
all Theorems generalize to symmetric Finsler metrics.
\item[3)]
M.~L.~Bialy and L.~Polterovich ~\cite{BP} and independently V.~Bangert~\cite{B1988} present a 
formalism for minimal geodesics on $T^2$ based on a special class of orbits for monotone twist maps.
For this setting they formulate the notion of rotation number. However, in general it is not 
possible to extend their formalism to non-minimal geodesics and arbitrary orbits of twist maps. \\
The study of monotone twist maps with variational methods and many results about their
properties mentioned in \cite{B1988} and \cite{BP} go back to J.~N.~Mather \cite{Mather} and 
independently to S.~Aubry and P.~Y.~Le~Daeron \cite{Aubry}.
In \cite{A-int} S.~B.~Angenent studies orbits of monotone twist maps and their intersection
properties in the case that these maps have vanishing topological entropy. He formulates
analogious results to parts of the statements of Theorem~I and Theorem~III.
\item[4)]
A different approach to understand the relation between the complexity of the geodesic flow
and the behavior of geodesics similar to the one in this paper is presented by S.~V.~Bolotin
and P.~H.~Rabinowitz in \cite{BR}.
\end{enumerate}
\end{remo}

The paper is organized as follows: In the second section we introduce the notion of the
topological entropy and discuss the for us relevant properties of the Curve Shortening
flow. In section three we show that vanishing topological entropy implies the non-existence
of self-intersections of lifted geodesics and the non-existence of contractible geodesics
on $T^2$. Section four deals with geometric conditions which imply via the Curve
Shortening flow positive topological entropy. In the case of vanishing topological entropy
we exclude the existence of oscillating geodesic rays. Under the same assumption, in section
five we prove the existence and regularity of the rotation number. In section six we study
the intersection properties of geodesics on the universal covering with their translates.
In the last section we provide a characterization for flatness of Riemannian metrics on $T^2$.
In particular, we show that no axis on the universal covering intersects its translates
iff the torus is flat.

%%%%%%%%%%%%%%%%%%%%%%%%%%%%%%%%%%%%%%%%%%%%%%%%%%%%%%%%%%%%%%%%%

\section{Topological entropy and curve shortening}

The topological entropy is invariant under topological conjugations and measures as
described in the introduction the exponential orbit complexity  by a single non-negative
number. The precise meaning becomes apparent in the following definition of topological
entropy introduced by R.~E.~Bowen~\cite{BO}.

\begin{dfn}[Topological entropy] \label{topent}
Let $(Y,d)$ be a compact metric space, $\phi^t: Y \to Y$ a continuous flow and
$d(\cdot,\cdot)_{T}$ the dynamical metric defined by
$d(v,w)_{T}:= \max_{0 \leq t \leq T}d(\phi^tv, \phi^tw)$ for all $v,w \in Y$.
We fix $\varepsilon>0$. A subset $F\subset Y$ is called $(\phi,\varepsilon)$-{\it separated}
set of $Y$ with respect to $T$, if for $x_1 \not=x_2 \in F$ it holds $d(x_1,x_2)_{T}> \varepsilon.$ \\
The topological entropy of $\phi^t$ is defined as
$$
h_{top}(g) = h_{top}(\phi)=\lim_{\varepsilon \to 0} \limsup_{T \to \infty}
\left(\frac{1}{T} \log r_T(\phi, \varepsilon)  \right)
. $$
Here $r_T(\phi, \varepsilon)$ denotes the maximal cardinality of any
$(\phi,\varepsilon)$-separated set of $Y$ with respect to $T$.
\end{dfn}
For more details and properties of the topological entropy see for example
\cite{KH} or \cite{Wo}.
\\
\\
To prove that certain geometric constellations imply
positive topological entropy of the geodesic flow on $ST^2$ we will
use the Curve Shortening flow on $T^2$. In the following we will give the precise
definition of this flow and state the properties relevant in this context.
\\
\\
Let $(M,g)$ be a Riemannian surface and
$$\Gamma=\{\gamma\;|\; \gamma: S^1 \to M \;\text{smooth immersed closed curve
}\}$$
a family of (parameterized) immersed smooth closed curves on $M$.
We consider a continuous local semi-flow $\Psi^t: \Gamma \to \Gamma$ with
$\Psi^t(\gamma)=: \gamma_t$ and $t \in [0, T_{\gamma})$ defined by
\begin{eqnarray} \label{csf}
\frac{\partial \gamma_t}{\partial t} = k_t N_t,
\end{eqnarray}
where $k_t$ denotes the geodesic curvature of $\gamma_t$ and $N_t$ its unit normal vector.
This evolution equation defined in~(\ref{csf}) is called the {\it Curve Shortening flow}.
\\
\\
From the following theorem we will derive the existence of closed geodesics:

\begin{theorem}[M.~A.~Grayson, see \cite{GO}]
\label{GO}
Let $M$ be a smooth Riemannian surface which is convex at infinity, i.e.\ the convex hull
of every compact subset is compact. Let $\gamma_0:S^1 \to M$ be a smooth curve, embedded
in $M$. Then, $\gamma_t:S^1 \to M$ exists for $t \in [0, T)$, for some $T>0$, satisfying
the evolution equation (\ref{csf}).\\
If $T$ is finite, then $\gamma_t$ converges to a point. If $T$ is infinite, then the
curvature of $\gamma_t$ converges to zero in the $C^{\infty}$ norm, i.e.\ there exists
a subsequence $t_n$ such that $\gamma_{t_n}$ converges to a closed geodesic.
\end{theorem}

The assumption that $M$ is convex at infinity ensures that the set of limit curves exists
since an evolving curve $\gamma_t$ cannot leave a compact set. In our special case of the
torus we will directly show that the embedded curve $\gamma_t$ stays in a compact set on
the universal covering $\Br^2$ of $T^2$ or in a compact set on some unbounded cylinder
$C$ for all $t \in [0,T)$. Then, by excluding that $\gamma_t$ converges to a point we
conclude the existence of a closed geodesic.
\\
As the length $l_t$ of the curve $\gamma_t$ fulfills $$\frac{d l_t}{dt} = - \int k_t^2(s) ds,$$
the length is a decreasing function of $t$. This is why this flow is called Curve Shortening.\\
An important fact is that embedded curves never become singular, unless they shrink
to a point, as proved by M.~A.~Grayson~\cite{GO}.

Furthermore we will apply the following consequence of the maximum-principle for parabolic
differential equations several times:

\begin{theorem}[S.~B.~Angenent, see \cite{A-II}]  \label{2seg}
Let $\gamma_0, \eta_0:[0,1] \to M$ be two curve-segments and $\{\gamma_t\;|\; 0<t<T\}$,
$\{\eta_t\;|\; 0<t<T\}$ solutions of Curve Shortening which satisfy
$$\partial \gamma_t \cap \eta_t = \partial \eta_t \cap \gamma_t= \emptyset$$
for all $t \in [0,T)$. Then the number of intersections of the solutions $\gamma_t$ and
$\eta_t$ is a finite and nonincreasing function of $t \in (0,T)$. The solutions intersect
only transversally except at a discrete set of times $\{t_j\}\subset (0,T)$, and at each
$t_j$ the number of intersections of $\gamma_t$ and $\eta_t$ decreases.
\end{theorem}

%%%%%%%%%%%%%%%%%%%%%%%%%%%%%%%%%%%%%%%%%%%%%%%%%%%%%%%%%%%%%%%%%%%%%%

\section{Contractible Closed Geodesics and Self-intersections}

In 1998 J.~Denvir and R.~S. MacKay \cite{DM} showed the following result as a
conclusion of their study on geodesically convex surfaces of negative Euler characteristic.

\begin{theorem}\label{scc}
Let $g$ be a Riemannian metric on $T^2$ with a simple closed contractible geodesic $c$.
Then $g$ has positive topological entropy.
\end{theorem}

This theorem allows us to conclude the following lemma:

\begin{lemma}\label{2-loops}
Let $g$ be a Riemannian metric on $T^2$ and $c:\Br \to \Br^2$ a lift of a geodesic
with two self-intersections such that $c(t_1)=c(t_2)$ and $c(t_3)=c(t_4)$ with
$t_1<t_2<t_3<t_4$. Then $g$ has positive topological entropy.
\end{lemma}

\begin{proof} We use similar ideas as S.~Angenent in~\cite{A-pre}.
Let $\gamma_0:S^1 \to \Br^2$ be a simple closed contractible and smooth curve such
that $\gamma_0$ and $c([t_1, t_4])$ do not intersect on $\Br^2$ and such that
$c[t_1, t_4]$ lies in the bounded connected component of $\Br^2 \setminus \gamma_0(S^1)$.
We apply the curve-shortening flow to $\gamma_0$ and $c$. Note that $c([t_1, t_4])$
is a constant solution. Assume, that $\gamma_t$ shrinks to a point. Then the solution
$\gamma_t$ has to pass $c([t_1, t_4])$ under the flow. Consider $\tilde t \geq 0$ with
$\gamma_{\tilde t}(S^1) \cap c([t_1, t_4]) \not= \emptyset$ such that $\tilde t$ is
the first time at which the segments meet. By Theorem \ref{2seg} new intersections
under the flow can only appear in the endpoints $c(t_1)$ or $c(t_4)$. Consider a small
$\varepsilon>0$ such that $t_1 + \varepsilon < t_{2}$ and $t_4 - \varepsilon>t_3$.
As $c(t_1)=c(t_2)$ and $c(t_3)=c(t_4)$ the curve $\gamma_{\tilde t}$ meets the
geodesic segment $c([t_1 + \varepsilon, t_4 - \varepsilon])$ also at
$s \in (t_1 + \varepsilon, t_4 - \varepsilon)$, in contradiction to Theorem \ref{2seg}.
Hence, as $\gamma_t$ cannot pass $c([t_1, t_4])$ it will not shrink to a point.\\
Furthermore, $\gamma_t$ will stay in a bounded set ${\mathcal{K}} \subset \Br^2$
for all $t$: We consider two non-equivalent primitive translation elements $\tau, \eta$
and two corresponding axes $\alpha_{\tau}$ and $\alpha_{\eta}$. For $k \in \Bn$ large
enough the curve $\gamma_0$ lies in the open strip between $\tau^k \alpha_{\eta}$ and
$\tau^{-k}\alpha_{\eta}$, and in the open strip between $\eta^k \alpha_{\tau}$
and $\eta^{-k}\alpha_{\tau}$.
\begin{figure}[h!]
\begin{center}
\psfrag{a}{$\tau^k\alpha_{\eta}$}
\psfrag{b}{$\eta^k\alpha_{\tau}$}
\psfrag{d}{$\tau^{-k} \alpha_{\eta}$}
\psfrag{c}{$\eta^{-k} \alpha_{\tau}$}
\psfrag{e}{$c$}
\psfrag{K}{$\mathcal{K}$}
\psfrag{f}{$\gamma_0$}
\includegraphics[scale=0.36]{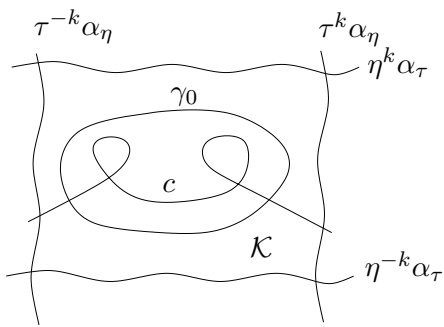}
\end{center}
\caption{Illustration of the two loops of $c$ and $\gamma_0$ in the proof of Lemma~\ref{2-loops}.}
\end{figure}
The intersection ${\mathcal{K}}$ of the closures of these strips defines a region from
which $\gamma_t$ cannot escape, because since $\gamma_0$ does not intersect any of
the bounding minimal geodesics, by Theorem~\ref{2seg} $\gamma_t$ will not intersect
them for all $t>0$.\\
Now we apply M.~A.~Grayson's Theorem~\ref{GO}: Since $\gamma_t$ will not shrink to a
point and stays in a compact set for all times, there exists a simple closed geodesic
on $\Br^2$. By Theorem \ref{scc} $g$ has positive topological entropy.\\
\end{proof}

In the sequel we will make use of the following standard recurrence theorem
(see for instance \cite{Wo}, page 157).
\begin{theorem}
Let $X$ be a manifold and $\mu$ a probability measure invariant under a continuous
flow $\phi^t: X \to X$. Then almost all $p \in X$ are recurrent, i.e.\ there exists a
sequence $t_n$ with $t_n \to \infty$ as $n \to \infty$ such that
$$
\lim\limits_{n \to \infty} \phi^{t_n}p =p \enspace .
$$
Furthermore, if $p$ is recurrent all points on the orbit $\phi^t(p)$ are recurrent
as well. If $\mu$ is positive on open sets, the set of recurrent points is dense.
\end{theorem}

\begin{theorem}\label{no-intersections}
Let $g$ be a Riemannian metric on $T^2$ with zero topological entropy. Then no lift
of a geodesic has a
self-intersection.
\end{theorem}

\begin{proof}
Let $c:\Br \to \Br^2$ be a lift of a geodesic and assume that $c$ has a self-intersection.
Then there exist $t_0 < t_1 \in \Br$ with $c(t_0)= c(t_1)$. Let $\dot c(t_0)=w$ denote
the initial condition of $c$ in $t_0$. Assume first that $c$ is recurrent. Then there
exists an increasing sequence of times $(t_n)_{n \in \Bn}$ and a sequence $\tau_n$ of
translation elements such that $D\tau_n\dot c(t_n)=w_n \to w$ for $n \to \infty$, see
Figure~\ref{int}. For $n$ large enough and $t_n$ much larger than $t_1$ there exist $t$
near $t_n$ and $s>0$ such that $\tau_n c(t)=\tau_n c(t+s)$. Hence, $c(t)=c(t+s)$.
As $t_0<t_1<t<t+s$ by Lemma~\ref{2-loops} the metric $g$ has positive topological entropy,
in contradiction to the assumption.
\begin{figure}[h!]
\begin{center}
\psfrag{c}{$c$}
\psfrag{ct}{$\tau_n c$}
\psfrag{w}{$w$}
\psfrag{wn}{$w_n$}
\includegraphics[scale=0.56]{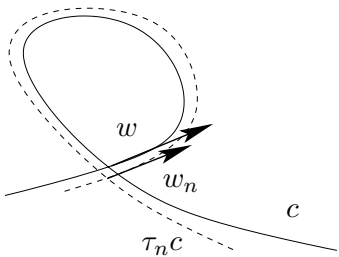}
\end{center}
\caption{Illustration of $c$ and $\tau_n c$ in the proof of Theorem~\ref{no-intersections}.} \label{int}
\end{figure}
Let $c$ be an arbitrary geodesic. As recurrent geodesics are dense,
there exists a sequence of recurrent geodesics $c_n$ with $\dot c_n(t_0)=w_n$ such
that $w_n$ converges to $w$. Since $c$ has a self-intersection, i.e.\ $c(t_0)=c(t_1)$,
the continuous dependency implies that for $n$ large enough the recurrent geodesic $c_n$
has a self-intersection. Then by Lemma~\ref{2-loops} the metric $g$ has positive
topological entropy in contradiction to the assumption.
\end{proof}

The following theorem is due to V.~Bangert in~\cite{B1981}
\begin{theorem} \label{bounded}
Let $g$ be a complete Riemannian metric on $\Br^2$. Then the existence of a bounded
geodesic ray $c:[0, \infty) \to \Br^2$ implies the existence of a simple closed geodesic.
\end{theorem}

\begin{remo}
We note that this result can be also obtained combining Curve Shortening with
arguments from topological dynamics. For more details about this approach
see~\cite{Glas} where a slightly weaker result has been obtained.
\end{remo}

\begin{cor} \label{bounded2}
Let $g$ be a Riemannian metric on $T^2$. Assume there exists a bounded geodesic ray $c$
on the universal covering $\Br^2$. Then the metric $g$ has positive topological entropy.
\end{cor}

%%%%%%%%%%%%%%%%%%%%%%%%%%%%%%%%%%%%%%%%%%%%%%%%%%%%%%%%%%%%%%%%%%%%%%%%%%%%

\section{Central Geometric Argument}

In this section we present a fundamental geometric constellation on the universal
covering of $T^2$ and prove that it implies positive topological entropy for the
Riemannian metric $g$ on $T^2$. In the further sections we will use this argument
several times.

\begin{dfn}
A continuous curve $c: I \to \Br^2$, for $I \subset \Br$, is called a broken geodesic
if there exists a finite set $W \subset I$ such that $c$ is geodesic on $I\setminus W$.
We call $W$ the set of vertices.
\end{dfn}

\begin{lemma}[Fundamental Lemma] \label{fundamentallemma}
Let $g$ be a Riemannian metric on $T^2$ and $\alpha: \Br \to \Br^2$ a
minimal axis of the translation element $\tau$. Let $c_1:[0, a] \to \Br^2$ and
$c_2:[0, b] \to \Br^2$ be two geodesic segments with endpoints on $\alpha$ and
$c_1((0, a)) \cap \alpha(\Br)=\emptyset$, $c_2((0, b)) \cap \alpha(\Br)=\emptyset$.
Assume that there exists a translation element $\eta$, with
$\eta\alpha(\Br) \cap \alpha(\Br)= \emptyset$ such that
$\eta \alpha(\Br) \tri c_1([0, a]) \not= \emptyset$ and
$\eta^{-1} \alpha(\Br) \tri c_2([0, b]) \not= \emptyset$. Then the metric $g$ has
positive topological entropy.
\end{lemma}

\begin{proof}
We choose $k \geq 1$ such that
$$\eta^k \alpha(\Br) \tri c_1([0, a]) \not= \emptyset \quad \text{but} \quad
\eta^{k+1} \alpha(\Br) \tri c_1([0, a]) = \emptyset.$$ Analogously we choose
$l \geq 1$ such that $$\eta^{-l} \alpha(\Br) \tri c_2([0, b])\not= \emptyset \quad
\text{but} \quad \eta^{-(l+1)} \alpha(\Br) \cap c_2([0, b])= \emptyset.$$
W.l.o.g.\ let $l \geq k$. Then $$\alpha(\Br) \cap \eta^{l+1} c_2([0, b])
= \emptyset \quad \text{but} \quad
\eta \alpha(\Br) \tri \eta^{l+1} c_2([0, b]) \not= \emptyset.$$
\begin{figure}[h!]
\begin{center}
\psfrag{a}{$\alpha$}
\psfrag{ka}{$\eta^{k} \alpha$}
\psfrag{1a}{$\eta\alpha$}
\psfrag{l1a}{$\eta^{(l+1)} \alpha$}
\psfrag{c2}{$c_1$}
\psfrag{c1}{$\eta^{(l+1)} c_2$}
\includegraphics[scale=0.3]{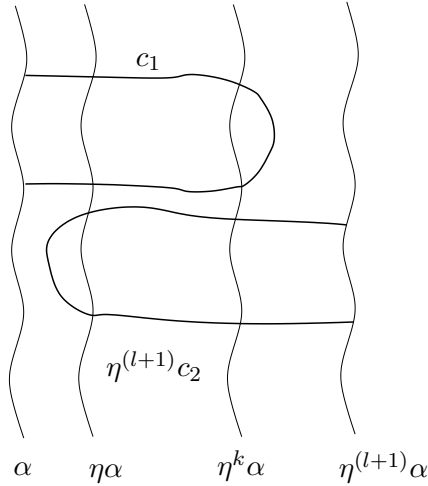}
\end{center}
\caption{Illustration of the translated segments $c_1$ and $\eta^{(l+1)} c_2$
in the proof of the Fundamental Lemma~\ref{fundamentallemma}.}
\end{figure}
By $S$ we denote the geodesic strip bounded by $\alpha$ and $\eta^{l+1} \alpha$.
In the next step we use the segments $c_1$ and $\eta^{l+1}c_2$ to construct
suitable broken geodesics $\sigma_1$ and $\sigma_2$ which intersect. We distinguish
two cases:
\begin{enumerate}
\item[Case 1)] There exists $\tilde n \in \Bz$ such that $\tau^{\tilde n} c_1$
and $\eta^{l+1}c_2$ intersect. In this case we will denote
$\sigma_1: = \tau^{\tilde n} c_1$ and $\sigma_2: = \eta^{l+1}c_2$.
\item[Case 2)]
For all $\tilde n \in \Bz$ the segments $\tau^{\tilde n} c_1$ and
$\eta^{l+1}c_2$ do not intersect. Using $c_1$ and $\eta \alpha$ we will construct
a broken geodesic segment $\sigma_1$ with endpoints on $\alpha$ which
intersects $\eta^{l+1}c_2$, see Figure~\ref{fig-ref4}.
\begin{figure}[h!]
\begin{center}
\psfrag{a}{$\alpha$}
\psfrag{c}{$c_1$}
\psfrag{tc}{$\tau^{n} c_1$}
\psfrag{ec}{$\eta^{l+1} c_2$}
\psfrag{ka}{$\eta^{k} \alpha$}
\psfrag{1a}{$\eta\alpha$}
\psfrag{l1a}{$\eta^{l+1} \alpha$}
\includegraphics[scale=0.3]{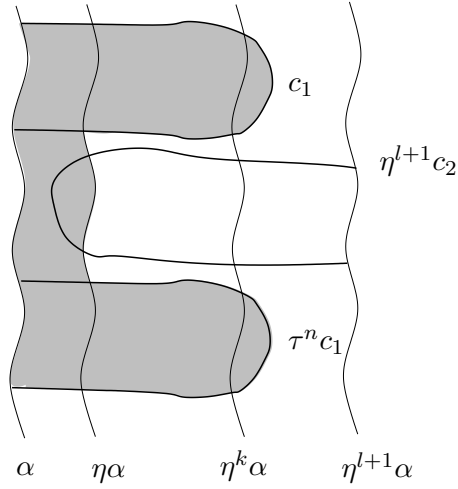}
\end{center}
\caption{The construction of intersecting broken geodesic segments $\sigma_1$
and $\sigma_2$ in the proof of the Fundamental Lemma~\ref{fundamentallemma}.} \label{fig-ref4}
\end{figure}
We choose $n \in \Bz$
such that on $\eta \alpha(\Br)$ the set of intersections
$\eta^{l+1}c_2 \tri \eta \alpha(\Br)$ lies
between the finite sets
$M_+=\tau^n c_1 \tri \eta \alpha(\Br)$ and
$M_-= c_1 \tri \eta \alpha(\Br)$.\\
Let $\eta \alpha([t_1,t_2])$ be the smallest segment of $\eta \alpha$
such that
$$\eta \alpha([t_1,t_2]) \cap M_+ \not= \emptyset
\not= \eta \alpha([t_1,t_2]) \cap M_-.$$
Let $A$ be the union of the bounded connected components of
$$S \setminus \left(c_1([0, a])
 \cup \tau^n c_1([0, a])
 \cup \eta \alpha([t_1,t_2]) \right).$$
We consider the broken geodesic $\sigma_1([0, \tilde a])$ endowed
with a new para\-metrization consisting of $\eta \alpha([t_1,t_2])$
and segments of $\tau^n c_1$ and $c_1$, such that it separates
$\bar A$ from $S \setminus \bar A$. By construction $\sigma_1([0,\tilde a])$
and $\sigma_2([0,b]) = \eta^{l+1} c_2([0, b])$ intersect.
\end{enumerate}
Let $m \in \mathbb{N}$ be large enough such that
\begin{eqnarray} \label{m2}
(\sigma_1 \cup \sigma_2) \cap \tau^{mi}(\sigma_1 \cup \sigma_2) = \emptyset
\quad \text{for all $i \in \Bz \setminus\{0\}$.}
\end{eqnarray}
For fixed $j \in \Bn$ and a finite sequence of $j$ elements
$(a_0,...,a_{j-1})=a(j)$ with $a_k \in \{1,2 \}$ we consider
the set
$$
A(a(j)):= \bigcup\limits_{k =0}^{j-1} \tau^{mk} \sigma_{a_k}\enspace.
$$
For each $a(j)$ let $\gamma_0^{a(j)} : \mathbb{R} \to S \setminus A(a(j))$
be a smooth curve such that $\tau^{mj}\gamma_0^{a(j)}=\gamma_0^{a(j)}$,
see Figure \ref{fig-ref5}.
\begin{figure}[h]
\begin{center}
\psfrag{S}{$S$}
\psfrag{c1}{$\tau^m \sigma_1$}
\psfrag{c2}{$\tau^m \sigma_2$}
\psfrag{tc1}{$\tau^{2m}\sigma_1$}
\psfrag{tc2}{$\tau^{2m} \sigma_2$}
\psfrag{ttc1}{$\tau^{3m} \sigma_1$}
\psfrag{tal}{$\alpha$}
\psfrag{ta}[c]{$\eta^{l+1} \alpha$}
\psfrag{t2c2}{$\tau^{3m} \sigma_2$}
\psfrag{g0}{$\gamma_0^{(a)}$}
\psfrag{1}{$\gamma_0^{(a)}$}
\psfrag{a}{$\alpha$}
\psfrag{C}[c]{$C_3$}
\psfrag{ga}{$\gamma_0$}
\includegraphics[scale=0.24]{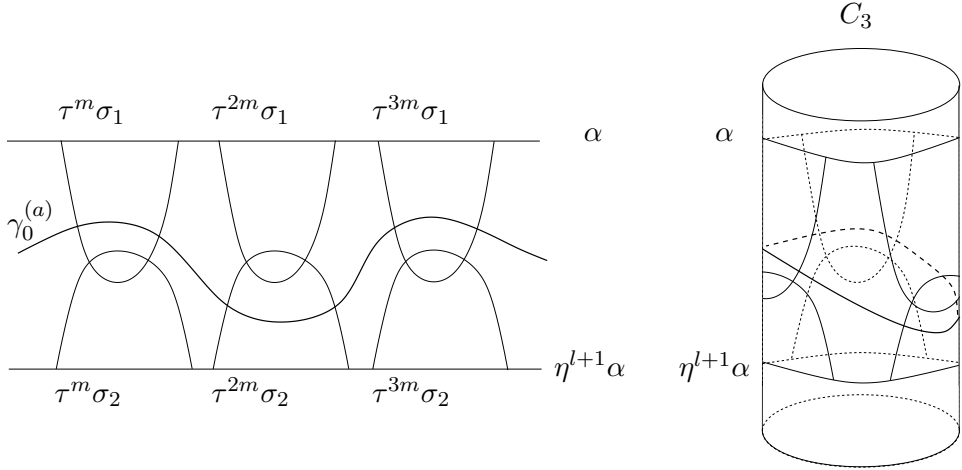}
\end{center}
\caption{Simplified illustration of the translates of $\sigma_1$,
$\sigma_2$, and $\gamma_0^{(a)}$ in the proof of the Fundamental
Lemma~\ref{fundamentallemma}.} \label{fig-ref5}
\end{figure}
Furthermore, we can choose these curves such that the length of
the segment of $\gamma_0^{a(j)}$ joining  $\gamma_0^{a(j)}(0)$
and $\tau^{jm}\gamma_0^{a(j)}(0)$ is less than $jb_2$ and larger than
$jb_1$ for some universal constants $0<b_1 < b_2$. Consider for each
$j \in \Bn$ the cylinder
$$
C_j: = \Br/\langle \tau^{mj} \rangle
$$
where $\langle \tau^{jm} \rangle$ is the subgroup of $\Bz^2$ generated
by $\tau^{jm}$. In particular, the projection of $S$ onto $C_j$ defines
an annulus. By assumption the curves $\gamma_0^{(a(j))}$ project to closed
curves $\gamma_0^{(a(j))}: \mathbb{R} \to C_j \setminus  B(a(j))$, where
$B(a(j))$ is the projection of $A(a(j))$ onto $C_j$. By construction the
interior angles at the vertices $w_i$ of the connected component of
$C_j \setminus  B(a(j))$ containing $\gamma_0^{(a(j))}$ are all less than $\pi$.
Applying the curve shortening flow to the smooth curve $\gamma_0^{(a(j))}$
we obtain smooth curves $\gamma_t^{a(j)}: \mathbb{R} \to C_j \setminus  B(a(j))$
which stay in $C_j$ and do not intersect $ B(a(j))$. This follows from
Theorem~\ref{2seg} and the fact that due to their homotopy class $\gamma_t^{a(j)}$
never become singular. Since the curves $\gamma_0^{a(j)}$ are not contractible,
M.~A.~Grayson's Theorem~\ref{GO} implies that the curvature of $\gamma_t^{a(j)}$
converges to zero. This yields the existence of a closed geodesic on $C_j$.  \\
Let $D_i$ be the bounded connected component of $S \setminus \sigma_i$ for
$i \in \{1,2\}$. Let $B(x, \varepsilon) \subset D_1 \cap D_2$ be a geodesic
ball with radius $\varepsilon>0$. Obviously, by construction the curves
$\gamma_0^{a(j)}$ do not intersect $\tau^{mk} B(x, \varepsilon)$ for
$0 \leq k \leq j-1$ for all sequences $a(j)$ and $\gamma_t^{a(j)}$ will not
intersect $\tau^{mk} B(x, \varepsilon)$ for all $0 \leq k \leq j-1$ and
$t \in (0, T)$ by the properties of the Curve Shortening flow.\\
By this construction for fixed $j$ and different sequences $a(j)$ we get $2^j$
different closed geodesics on $C_j$ of length between $b_1j$ and $b_2j$.
We project the constructed closed geodesics onto $C_1$. Identifying $\alpha$
and $\eta^{l+1} \alpha$ we get a torus $\tilde T^2$ on which we still have at
least $\frac{2^j}{j}$ different closed geodesics. By construction the initial
conditions of these closed geodesics form a $(\phi, \varepsilon)$-separated set
which grows exponentially for $j \to \infty$. By Definition~\ref{topent} this implies
positive topological entropy for the Riemannian metric on $\tilde T^2$.
As $\tilde T^2$ is a finite cover of $T^2$ also the Riemannian metric $g$ on
$T^2$ has positive topological entropy.
\end{proof}

\begin{remo}
For the proof of the Fundamental Lemma it is not necessary that $\alpha$ is minimal.
It suffices the condition $\alpha \cap \eta^{k+1} \alpha= \emptyset$ for the
special $k\in \Bz$ used in the proof.
\end{remo}

\begin{cor} \label{closed}
Let $(T^2,g)$ be a Riemannian torus. The existence of a contractible closed
geodesic on $T^2$ implies the existence of non-primitive axes for all translation
elements $\tau$.
\end{cor}

\begin{proof}
Considering an arbitrary translation element and replacing in the proof of
the Fundamental Lemma~\ref{fundamentallemma} the segments $\sigma_1$ and
$\sigma_2$ by a closed geodesic we can use a similar construction, as presented
starting in (\ref{m2}), to produce a non-primitive axis.
\end{proof}

Using the recurrence of the geodesic flow on $T^2$ we obtain the following
refinement of the Fundamental Lemma.

\begin{theorem}\label{1g}
Let $(T^2,g)$ be a Riemannian torus. Let $c:\Br \to \Br^2$ be a geodesic and
$\alpha$ a minimal axis such that $c(a), c(b) \in \alpha (\Br)$ for $a < b$.
If there exists a translation element $\tau$ with $\tau(\alpha( \Br)) \cap \alpha( \Br)= \emptyset$
and $\tau(\alpha( \Br)) \tri c( [a,b])\not= \emptyset$, then $g$ has positive
topological entropy.
\end{theorem}

\begin{proof}
Let $c:\Br \to \Br^2$ be a geodesic and $\alpha$ a minimal axis such that
$c(a), c(b) \in \alpha (\Br)$ for $a < b$ and $\tau$ the translation element
with $\tau(\alpha( \Br)) \cap \alpha( \Br)= \emptyset$  and
$\tau(\alpha( \Br)) \tri c( [a,b])\not= \emptyset$.
We will now prove that this implies positive topological entropy.\\
Consider the two halfplanes given by the connected components of
$ \Br^2 \setminus \tau(\alpha( \Br))$ and denote by $A$ the one containing
$\alpha(\Br)$ and by $A'$ the other one. By assumption there exists $t_0 \in [a,b]$
such that  $c(t_0) \in A'$.
\\
We can assume that $c[b, \infty)$ does not intersect $\tau(\alpha(\Br))$ since
this would immediately imply positive topological entropy by the Fundamental
Lemma~\ref{fundamentallemma}. We can also assume that $c$ is recurrent, since
the recurrent geodesics are dense. In particular,
there exists an increasing sequence $(t_n)_{n \in \Bn}$ and a sequence $\tau_n$
of translates such that $D\tau_n \dot c(t_n) \to \dot c(t_0)$ for $n \to \infty$.
But then for $n$ large enough $c([b, t_n])$ must intersect the translates
$\tau_n \tau \alpha( \Br) $ and  $\tau_n \alpha( \Br) $. Then, $c([b, t_n])$
fulfills the assumption of the Fundamental Lemma~\ref{fundamentallemma} and
hence, $g$ has positive topological entropy and there exist non-primitive axes
of $\tau$.
\end{proof}

\begin{cor} \label{2g}
Let $(T^2,g)$ be a Riemannian torus.
Let $c:[0 , \infty) \to \Br^2$ be a geodesic ray on the universal covering
$\Br^2$ which intersects a minimal axis $\alpha$ and disjoint translate
infinitely often. Then $g$ has positive topological entropy.
\end{cor}

\begin{lemma} \label{osc}
Let $g$ be a Riemannian metric on $T^2$. Then the existence of an
oscillating geodesic ray $c$ on $\Br^2$ implies positive topological
entropy for the metric $g$.
\end{lemma}

\begin{proof}
Due to the definition of an oscillating geodesic ray there exist sequences
$t_n, s_n$ tending to infinity and a compact set $K \subset \Br^2$ such that
$c(s_n) \in K$ and $\|c(t_n)\| \to \infty$ for $n \to \infty$. W.l.o.g.\ we can
assume that
$$
\lim\limits_{n \to \infty} \frac{c(t_n)}{\|c(t_n)\|} = z
$$
for some $z \in S^1$. Let
$\alpha: \Br \to \Br^2$ be a minimal axis such that $\pm \delta(\alpha) \not= z$.
Furthermore, we can choose $\alpha$ and a disjoint translate $\alpha'$ in such a
way that $K$ and $\{ c(t_n) \mid n \ge n_0 \}$ are contained in the two
different halfplanes of $\Br^2 \setminus (\alpha(\Br) \cup \alpha'(\Br))$
for sufficiently large $n_0$. This implies that the geodesic ray
$c:[0 , \infty) \to \Br^2$ will intersect $\alpha$ and $\alpha'$ infinitely
often. By Corollary~\ref{2g} the metric $g$ has positive topological entropy.
%and there exist non-primitive axes on the universal covering $\Br^2$.
\end{proof}

%%%%%%%%%%%%%%%%%%%%%%%%%%%%%%%%%%%%%%%%%%%%%%%%%%%%%%%%%%%%%%%%%%%%%%%%%%%%%%%%

\section{Existence and Regularity of the Rotation Number}
In the proof of Theorem~I we will need the following property of minimal geodesics:
\begin{remo}
As already mentioned, the study of minimal geodesics on $T^2$ goes back to
H.~M.~Morse~\cite{Morse}, G.~A.~Hedlund~\cite{Hedlund}, and V.~Bangert~\cite{B1988}.
Central results are that for each $r \in \Br \cup \{\infty\}$ there exists
a minimal geodesic $c:\Br \to \Br^2$ with asymptotic direction $r$ and that
furthermore there exists a constant $D>0$, such that for each minimal geodesic
$c: \Br \to \Br^2$ there exists a Euclidean line $l_c$, and for each Euclidean
line $l_c$ there exists a minimal geodesic $c$ such that
$$d (l_c, c(t)) \leq \frac{D}{2}, \quad \text{for all $t \in \Br.$}$$
Furthermore, we can assume that $D$ is larger than the diameter of the
fundamental domain. As shown by V.~Bangert~\cite{B1988}, for irrational rotation
numbers the set of minimal geodesics with this rotation number is totally ordered,
i.e., all these minimal geodesics have pairwise no intersections with each other.
For the set of minimal geodesics with a fixed rational rotation number the
subset of axes is ordered. Two minimal axes with the same rotation number bounding
a strip containing no further minimal axes are called neighboring minimals.
\end{remo}

\begin{proof}[Proof of Theorem~I]
Combining Corollary~\ref{bounded2}, Lemma~\ref{osc}, and Theorem~\ref{no-intersections}
we conclude that in the case of vanishing topological entropy all geodesic rays
are escaping and have no self-intersections on the universal covering $\Br^2$.\\
If for $c^+$ the asymptotic direction does not exist, there exist two
accumulation points $z_1 \not= z_2 \in S^1$ for the quotient
$\frac{c^+(t)}{\|c^+(t)\|}$ as $t$ tends to $\infty$. Choose a rational direction
$z \in S^1$ which lies in the connected component of $S^1 \setminus\{z_1,z_2\}$,
which is not larger than a half-circle. Let $\alpha$ be a minimal axis of the
primitive translation element $\tau$ with $\delta(\tau)=z$. Then the geodesic
$c:[0, \infty) \to  \Br^2$ must intersect $\alpha(\Br)$ and a given disjoint
translate $\eta(\alpha(\Br))$ with $[\eta] \not= [\tau]$ (which is a minimal axis
of $\tau$ as well) infinitely often. But then, by Corollary~\ref{2g} the metric
$g$ has positive topological entropy in contradiction to the assumption.\\
Assume that $\delta(c^+) \not= - \delta(c^-)$.
Choose $z \in S^1$ rational such that $\pm z$ both lie in one
connected component of $S^1 \setminus \{\delta(c^+), \delta(c^-)\}$. Let
$\alpha$ be an axis with $\delta(\alpha)=z$ which intersects $c$. Then $\alpha$
divides $\Br^2$ in two halfplanes. There is precisely one of those halfplanes
denoted by $A$ for which there exists $t_0 >0$ such that $c(t) \in A$ for $|t| \ge t_0$.
Choose a disjoint translate $\tau \alpha$ which is contained in $A$ as well.
Then there  exist $a <b$ such $c(a) , c(b) \in \tau(\alpha( \Br))$ and
$c[a,b] \tri \alpha(\Br) \not= \emptyset$ which contradicts Theorem~\ref{1g}.\\
We will now show that there exists a Euclidean strip $S(c)$ with direction
$\delta(c^+)$ such that $c(\Br) \subset S(c)$. We already know that the
asymptotic directions exist and that $\delta(c^+)=-\delta(c^-)$.
Assume that $c$ is not contained in a Euclidean strip. Consider an arbitrary
minimal geodesic with asymptotic direction $\delta(c^+)$. As minimal geodesics
are accompanied by Euclidean lines with the same direction and as $c$ is not
bounded by any strip with direction $\delta(c^+)$, c also intersects an infinite
number of translates of this geodesic. Consider $D>0$ as introduced in the
previous remark. Choose a translation element $\tau$, a time $t_0 >0$ and a
minimal geodesic $b$ with asymptotic direction $\delta(c^+)$ such that
$c[0 , t_0]$ crosses $b$ and $\tau b$, $c(0)$ and $c(t_0)$ lie in different
halfplanes of $\Br^2 \setminus (b(\Br) \cup \tau b(\Br))$ and such that there
fits a Euclidean strip $R$ with direction $\delta(c^+)$ and width $4D$ between
$b$ and $\tau b$, see Figure~\ref{fig-ref6}.
\begin{figure}[h!]
\begin{center}
\psfrag{a}{$c(0)$}
\psfrag{f}{$R$}
\psfrag{b}[r]{$c(t_0)$}
\psfrag{c}{$S$}
\psfrag{d}{$b$}
\psfrag{e}{$\tau b$}
\psfrag{g}{$c$}
\psfrag{j}{$\alpha$}
\psfrag{i}{$\eta \alpha$}
\includegraphics[scale=0.55]{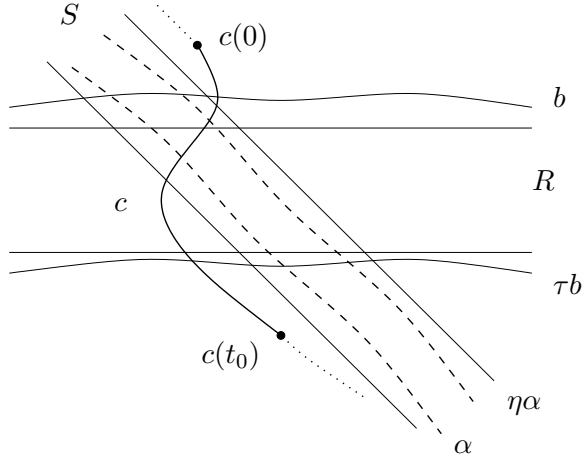}
\end{center}
\caption{Illustration of the argumentation in the second part of the proof of Theorem~I.} \label{fig-ref6}
\end{figure}
Then there exists a Euclidean strip $S$ with width $3D$ and a rational
asymptotic direction $z$ near $\delta(c^+)$ on $S^1$ but $z \not= \delta(c^+)$
such that $c(0)$ and $c(t_0)$ lie in different halfplanes of $\Br^2 \setminus S$.
As $S$ has width $3D$, by the previous remark there exists a minimal axis $\alpha$
and a disjoint translate $\eta \alpha$ with asymptotic direction $\delta(\alpha)=z$
and $(\alpha(\Br) \cup \eta \alpha(\Br)) \subset S$. Furthermore, the minimal axes $\alpha $
and $\eta\alpha$ intersect $c^+$. Since $\delta(\alpha^+)=z \not= \delta(c^+)$,
the ray $c^+$ has to intersect $\alpha $ and $\eta\alpha$ a second time which
implies positive topological entropy by Theorem~\ref{1g}.
\end{proof}
Now we like to prove regularity properties of the rotation number. We begin with
the following lemma which holds for all Riemannian metrics.

\begin{lemma} \label{surj}
For all $r \in \Br \cup \{\infty\}$ and $x \in T^2$ there exists $v=v(r) \in ST^2$
such that the corresponding geodesic $c_v:\Br \to T^2$ fulfills $c_v(0)=\pi(v)=x$,
$\dot c_v(0)=v$ and for the lift of $c_v$ on the universal covering it holds $\rho(c_v^+)=r$.
Let $E_x=\{v \in S_xT^2\;|\; \rho(v) \text{ exists }\}$. Then, for fixed $x \in T^2$ the map
$$f_x: E_x \to \Br \cup \{\infty\} \quad \text{with} \quad v \mapsto \rho(v)$$
is surjective.
\end{lemma}

\begin{proof}
We fix $r \in \Br \cup \{\infty\}$ and $x \in T^2$. By the remark at the beginning
of this section there exists a minimal geodesic $\gamma$ with rotation number $r$.
We assume that $x \notin \gamma(\Br)$, otherwise we set $\gamma=c$ after a
reparameterization such that $c(0)=x$ and so we choose $v=\dot c(0)$. \\
We consider minimal geodesic segments $c_t$ connecting $x$ and $\gamma(t)$ on the
universal covering such that $c_t(0)=x$ with $\dot c_t(0)=:w_t \in S_x \Br^2$.
As $S_x\Br^2$ is compact there exists a sequence $t_n$ with $t_n \to \infty$ such
that $w_{t_n}$ converges to $w$. By its minimality the segment $c_t$ intersects
$\gamma(\Br)$ only once in $\gamma(t)$.
\\
Let $\tau \gamma$ be a translate of $\gamma$ such that $x$ lies in the bounded strip
$S$ between $\gamma(\Br)$ and $\tau \gamma(\Br)$. Repeating the minimality arguments
we conclude that $c_t$ and $\tau \gamma(\Br)$ have no intersections for all $t$.
By the continuous dependence on the initial conditions also the limit geodesic ray
$c^+$ with $c^+(0)=x$ and $\dot c^+(0)=w$ will not intersect $\tau \gamma(\Br)$ and
$\gamma(\Br)$.\\
Hence, the geodesic ray $c^+$ lies in the geodesic strip $S$ with rotation number $r$.
This implies the existence of the rotation number and even that $\rho(c^+)=r$.
We extend $c^+$ on $T^2$ to the geodesic $c: \Br \to T^2$ with $c(0)=x$ and $\dot c(0)=w$.
The choice $v=w$ fulfills the required properties.
\end{proof}

\begin{proof}[Proof of Theorem~II]
By Lemma~\ref{surj} the rotation number is surjective. We have to show its continuity.
As the topological entropy vanishes, by Theorem~I the asymptotic directions $\delta(v)$
and $\delta(-v)$ exist and it holds $\delta(v)=- \delta(-v)$ for all $v \in S \Br^2$.
Here, we think of the asymptotic directions as a function on the unit tangent bundle
of $\Br^2$. Assume that $\delta$ is not continuous at $v_0 \in S\Br^2$. Let $c$ be
the geodesic with the initial condition $\dot c(0)=v_0$ and let $\delta(c^+)=z$.
As $\delta$ is not continuous at $v_0$, there exists a sequence $v_n \to v_0$ such
that $\delta(v_n)$ does not converge to $z$.
Since $S^1$ is compact there exists a subsequence $v_{n_k}$ with
$\delta(v_{n_k}) \to \tilde z \not= z$. We will again denote this subsequence by $v_n$
and the geodesics corresponding to $v_n$ by $c_n$, i.e.\ $\dot c_n(0)=v_n$.\\
Consider the two Euclidean rays $r: [0, \infty) \to \Br^2$ and $\tilde r: [0, \infty) \to \Br^2$
given by $r(t) = tz $ and $\tilde r(t) = t \tilde z$. Choose  a minimal axis
$\gamma: \Br \to \Br^2$ and a disjoint translate $\gamma': \Br \to \Br^2$ such that
the two halfplanes $H_1$ and $H_2$ which are two of the three connected components of
$\Br^2 \setminus (\gamma(\Br) \cup \gamma'(\Br))$ have the following properties:
First, there exists an $\varepsilon>0$ such that $B(c(0), \varepsilon) \subset H_1$.
Second, the sets $H_2 \cap r[0, \infty)$ and $H_1 \cap \tilde r[0, \infty)$ are
unbounded, see Figure~\ref{fig-ref7}.
\begin{figure}[h!]
\begin{center}
\psfrag{a}{$\tilde r$}
\psfrag{b}{$r$}
\psfrag{c}{$c$}
\psfrag{d}{$c_n$}
\psfrag{e}{$\gamma$}
\psfrag{f}{$\gamma'$}
\psfrag{g}{$H_1$}
\psfrag{h}{$H_2$}
\psfrag{c0}{$c(0)$}
\psfrag{cnt}[c]{$c_n(t_1)$}
\psfrag{cns}{$c_n(t_2)$}
\includegraphics[scale=0.7]{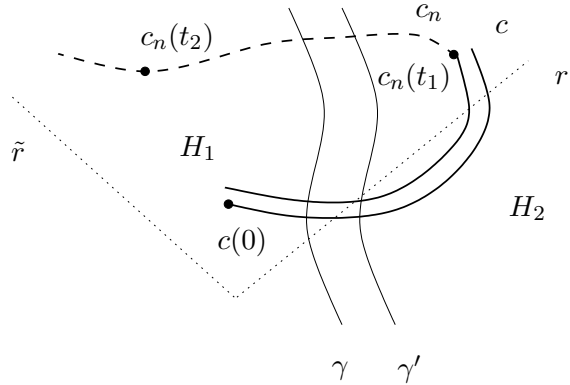}
\end{center}
\caption{Illustration of the geodesics $c$ and $c_n$ in the proof of Theorem~II.} \label{fig-ref7}
\end{figure}
Since $\delta(c^+)=z$ and $B(c(0), \varepsilon) \subset H_1$, the continuous
dependence on the initial conditions implies the existence  of  $n_0 \in \mathbb{N}$
such that for all $n \ge n_0$
we have $c_n(0) \in H_1$ and $c_n(t_1) \in H_2$ for a suitable $t_1 > 0$. Moreover,
$\delta(v_{n}) \to \tilde z$ as $n \to \infty$ yields the existence of  $n_1 \ge n_0$
such that the set $H_1 \cap c_n[0 , \infty)$ is unbounded for each $n \ge n_1$.
In particular, for a given $n \ge n_1$ we will find $t_2 \ge t_1$ such that $c_n(t_2) \in H_1$.
Such a geodesic $c_n$  crosses the pair of minimal axes $\gamma$ and $\gamma'$ at least
twice which implies by Theorem~\ref{1g} positive topological entropy in contradiction to
the assumption in Theorem~II. Hence, the asymptotic direction is continuous and its
continuity implies the continuity of the rotation number.
\end{proof}

\section{Further Conditions for Vanishing Topological Entropy}
In this section we show that vanishing topological entropy implies strong restrictions
of the intersections of geodesics on the universal cover with their translates.

\begin{proof}[Proof of Theorem~III]
As $h_{top}(g)=0$, by Theorem~I the asymptotic direction $\delta(c)$ exists for all
$c$ and for each $c$ there exists a Euclidean strip $S(c)$ bounding $c$.
Assume there exists a geodesic ray $c:[0, \infty) \to \Br$ with $\#I(c) \geq 2$.
Consider a translation element $\tau$ with $\#\{c \tri \tau c\}= \infty$ and
$\delta(\tau) \not=\pm \delta(c)$, e.g.\ it does not leave $S(c)$ invariant.
Let $k \in \Bz$ be large enough such that the distance between $S(c)$ and
$\tau^k(S(c))$ is larger than $4D$ with $D>0$ introduced in the remark at the
beginning of the previous section. We denote the Euclidean strip between $S(c)$
and $\tau^k (S(c))$ by $E$. For each $l \in \{0,\dots, k\}$ we denote the connected
component of $\Br^2 \setminus \tau^l c(\Br)$ which contains $\tau^{k+l} c(\Br)$ by
$A_{\tau^l c}$ and by $B_{\tau^l c}$ the other one. \\
As $c(\Br) \subset S(c)$ it follows that $\tau^k c(\Br) \subset \tau^k(S(c))$.
Consider $t_0 \in [0, \infty)$. As $\#\{c \tri \tau c\}= \infty$ there exist
$t_1,\tilde t_1>t_0$ with $c(t_1)=\tau c(\tilde t_1)$ such that $\tau c$ passes in
$c(t_1)$ from $B_c$ to $A_c$. Analogously there exist $t_2,\tilde t_2>\max\{t_1, \tilde t_1\}$
with $\tau c(t_2)=\tau^2 c(\tilde t_2)$ such that $\tau^2 c$ passes in $\tau c(t_2)$
from $B_{\tau c}$ to $A_{\tau c}$. By this construction we get a broken geodesic segment
$$ \alpha[t_0, t_{k}]:=c[t_0,t_1] \cup \tau c[\tilde t_1, t_2] \cup
\tau^2 c[\tilde t_2, t_3]\cup \dots \cup \tau^{k-1} c[\tilde t_{k-1}, t_{k}]$$
connecting $c$ and $\tau^k c$. Consider $s_0>t_{k}$ and analogously to the previous
construction there exist $s_1,\tilde s_1>s_0$ with $\tau^k c(s_1)=\tau^{k-1} c(\tilde s_1)$
such that $\tau^{k-1} c$ passes in $\tau^k c(s_1)$ from $A_{\tau^k c}$ to $B_{\tau^k c}$.
Analogously we get the broken geodesic segment
$$ \beta_1[s_0, s_{k}]:=\tau^k c[s_0,s_1] \cup \tau^{k-1} c[\tilde s_1, s_2]
\cup \tau^{k-2} c[\tilde s_2, s_3]\cup \dots \cup  \tau c[\tilde s_{k-1}, s_{k}].$$
Gluing these two geodesic segments we obtain the broken geodesic segment
$$V_1:=\alpha([t_0, t_{k}]) \cup \tau^k c([\tilde t_{k}, s_0]) \cup \beta_1([s_0, s_{k}])$$
which connects $S(c)$, $\tau^k(S(c))$, and $S(c)$. We call the unbounded connected
component of $\Br^2 \setminus\{S(c) \cup V_1\}$ which is not a Euclidean halfplane the
exterior of $V_1 \cup S(c)$. By construction the exterior angles of $V_1 \cup S(c)$ are
smaller than $\pi$. Analogously we construct a broken geodesic $V_2$ connecting $\tau^k(S(c))$,
$S(c)$, and $\tau^k(S(c))$ such that the exterior angles of $V_2 \cup \tau^k S(c)$ are
smaller than $\pi$. Consider two minimal axes $\xi_1$ and $\xi_2$ of $\tau$ bounding a
geodesic strip $G$ such that $V_1,V_2 \subset G$. As the distance of $S(c)$ and
$\tau^k(S(c))$ is larger than $4D$ consider three ordered translates of minimal axes
denoted by $\alpha_1$, $\alpha_2$, and $\alpha_3$ with a rational asymptotic direction
near $\delta(c^+)$ such that $\alpha_i \cap G \subset E$ for all $i \in \{1,2,3\}$.
Then, there exist subsegments $c_1$ and $c_2$ of $V_1$ and $V_2$ with endpoints on
$\alpha_2$ intersecting $\alpha_1$ or $\alpha_3$, respectively. These broken geodesic
segments $c_1$ and $c_2$ fulfill by construction the assumptions of the segments in the
Fundamental Lemma~\ref{fundamentallemma}.
\begin{figure}[h]
\begin{center}
\psfrag{a}{$V_1$}
\psfrag{b1}{$ $}
\psfrag{b2}{$ $}
\psfrag{g}{$V_2$}
\psfrag{E}{$E$}
\psfrag{G}{$G$}
\psfrag{x1}{$\xi_1$}
\psfrag{x2}{$\xi_2$}
\psfrag{sc}[c]{$S(c)$}
\psfrag{a1}{$\alpha_1$}
\psfrag{a2}{$\alpha_2$}
\psfrag{a3}{$\alpha_3$}
\psfrag{tsc}[c]{$\tau^k (S(c))$}
\includegraphics[scale=0.35]{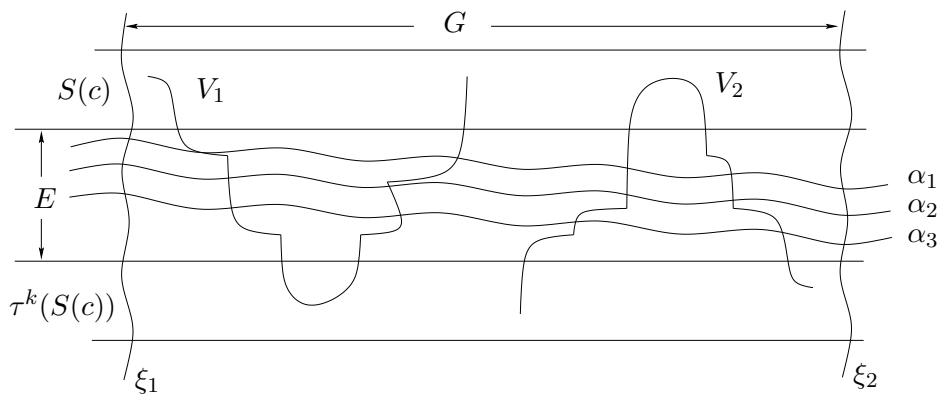}
\end{center}
\caption{Illustration of the broken geodesic segments $V_1$ and $V_2$ in the proof of Theorem~III.}
\end{figure}
Hence, $h_{top}(g)>0$ in contradiction to the assumption.
Furthermore, by construction we conclude, that geodesics $c$ with irrational
rotation number intersect their translates $\tau c$ only a finite number of times.
Geodesics $c$ with rational rotation number intersect their translates $\tau c$ an
infinite number of times at most for $\tau$ with $\delta(\tau)=\delta(c^+)$.
\end{proof}

%%%%%%%%%%%%%%%%%%%%%%%%%%%%%%%%%%%%%%%%%%%%%%%%%%%%%%%%%%%%%%%%%%%%%%%%%%%%%%%%%%%%%%%%%%%%%%%

\section{Characterization of Flatness}

In order to prove Theorem~IV we first prove the following lemma due to V.~Bangert
which is of independent interest.

\begin{lemma}\label{lemmaflat}
Let $(T^2, g)$ be a Riemannian torus. Suppose there exists a primitive translation element $\tau$
such that its minimal axes do not foliate $\mathbb{R}^2$. Then there exists $k \geq 2$ and a
non-primitive axis of $\tau^k$.
\end{lemma}

\begin{proof}
If the minimal axes of $\tau$ do not foliate $\mathbb{R}^2$, then there exist two neighboring
minimal axes $\tilde c_1, \tilde c_2: \mathbb{R} \to \mathbb{R}^2$ of $\tau$. Consider the cylinder
$\mathbb{R}^2 \diagup \langle \tau \rangle =C$ and the projections $c_1, c_2:\mathbb{R} \to C$
of $\tilde c_1$ and $\tilde c_2$ which are closed geodesics of equal length $\ell$. By assumption
all closed curves between $c_1$ and $c_2$ in the same homotopy class have length strictly larger
than $\ell$. Choose locally convex neighborhoods $C_1$ and $C_2$ about $c_1$ and $c_2$. (Take for
instance minimal geodesic loops homotopic to $c_1$ and $c_2$, respectively, close to $c_1$ and $c_2$.)
There exists $\varepsilon>0$ such that for all closed curves $\gamma$ homotopic to $c_1$ with
$L(\gamma) \leq \ell + \varepsilon$ we obtain that
$$\gamma \text{ is contained in } C_1 \cup C_2 \enspace. \qquad \qquad (\ast)$$
Using the methods in \cite{B2}, pp. 87/88 we obtain a constant $A>0$ such that for all $n>0$
there exists a closed geodesic $a_n$ between $c_1$ and $c_2$ with the following properties:
\begin{enumerate}
\item[(i)]
$a_n$ is homotopic to $c_1:[0, n\ell] \to C$
\item[(ii)]
$L(a_n) \leq n \ell + A$
\item[(iii)]
$a_n$ is not contained in $C_1 \cup C_2$
\end{enumerate}
Choose $n \geq 2$ large enough such that $\frac{A}{n} < \varepsilon$. Then there does not exist a curve
$b_n$ homotopic to $c_1:[0, \ell] \to C$ whose $n$-th iterate is equal to $a_n$. Otherwise using
property (ii),
$$L(b_n)=\frac{1}{n}L(a_n) \leq \ell + \frac{A}{n} \leq \ell + \varepsilon.$$
But by ($\ast$) this would imply that $b_n$ and hence $a_n$ is contained in $C_0 \cup C_1$ which
contradicts property (iii). Hence, $a_n$ is a non-primitive axis of $\tau^n$.
\end{proof}

\begin{proof}[Proof of Theorem~IV]
Obviously the flatness of the metric on $T^2$ implies that no periodic geodesic (and in fact no
geodesic) intersects its translates. Conversely, assume that no axis intersects its translates.
Then Lemma~\ref{lemmaflat} implies that for each primitive translation element $\tau$ the
associated minimal axes foliate $\mathbb{R}^2$. Since all axes are minimal the result of
N.~Innami in \cite{I} implies that $(T^2,g)$ does not have conjugate points. Then by E.~Hopf's
theorem \cite{H} the metric $g$ is flat.
\end{proof}

{\it Acknowledgement.} The authors like to thank Sigurd Angenent for explaining to us the essential features of the Curve Shortening flow which is a crucial technique in this paper. We are also grateful to Victor Bangert for explaining to us his variational methods which led to Theorem~IV.

\end{document}